\documentclass{amsart}
\usepackage{amssymb,amsmath,amsrefs}
\usepackage[colorlinks=true]{hyperref}
\usepackage[T1]{fontenc}
\hypersetup{urlcolor=blue, citecolor=green}
\usepackage[dvips]{graphicx}
\usepackage{color}
\usepackage{subcaption}
\usepackage{float}
\usepackage{tabularx}

\graphicspath{{Imagens/}}
\theoremstyle{plain}

\newtheorem*{conj*}{Conjecture}
\newtheorem*{cor*}{Corollary}
\newtheorem{theorem}{Theorem}[section]

\newtheorem{proposition}[theorem]{Proposition}

\newtheorem{lemma}[theorem]{Lemma}

\newtheorem{question}{Question}

\theoremstyle{definition}
\newtheorem*{def*}{Definition}
\newtheorem{remark}[theorem]{Remark}

\newtheorem{definition}[theorem]{Definition}

\newcommand{\C}{\mathcal C}

\newcommand{\SC}{{\mathcal C}}

\newcommand{\SO}{{\mathcal O}}
\newcommand{\SP}{{\mathcal P}}

\newcommand{\ga} {\gamma}

\newcommand{\Z}{\mathbb{Z}}
\newcommand{\N}{\mathbb{N}}
\newcommand{\R}{\mathbb{R}}
\newcommand{\eps}{\varepsilon}

\newcommand{\dist}{\operatorname{\textit{d}}}

\newcommand{\diam}{\operatorname{diam}}

\DeclareMathOperator{\Per}{Per}

\author{Bernardo Carvalho*}
\author{Piotr Oprocha}
\author{Elias Rego}

\title[Continuum-wise hyperbolicity]{Continuum-wise hyperbolicity, periodic shadowing, and measures of maximal entropy}

\begin{document}

\renewcommand{\thefootnote}{}

\footnote{2020 \emph{Mathematics Subject Classification}: Primary 37D20; Secondary 37C05.}

\footnote{\emph{Key words and phrases}: Periodic shadowing, periodic specification, cw-hyperbolicity, measures of maximal entropy.}

\footnote{\emph{* Corresponding Author}}

\renewcommand{\thefootnote}{\arabic{footnote}}
\setcounter{footnote}{0}

\begin{abstract}
We prove that cw-hyperbolic homeomorphisms with jointly continuous stable/unstable holonomies satisfy the periodic shadowing property and, if they are topologically mixing, the periodic specification property. We discuss difficulties to adapt Bowen's techniques to obtain a measure of maximal entropy for cw-hyperbolic homeomorphisms, exhibit the unique measure of maximal entropy for Walter's pseudo-Anosov diffeomorphism of $\mathbb{S}^2$, and prove it can be obtained, as in the expansive case, as the weak* limit of an average of Dirac measures on periodic orbits. As an application, we exhibit the unique measure of maximal entropy for the homeomorphism on the Sierpi\'nski Carpet defined in \cite{BO}, which does not satisfy the specification property.
\end{abstract}

\maketitle

\section{Introduction}

Hyperbolicity is a central notion in the study of chaotic dynamical systems, which originated from  works of Anosov \cite{A} and Smale \cite{S} as a source of chaotic dynamics (on compact spaces). Hyperbolic systems have reasonably well-understood dynamics from both topological and statistical aspects. Among many important results is the Spectral Decomposition Theorem \cite{S} which decomposes the non-wandering set as a finite number of disjoint elementary sets that are compact and periodic sets satisfying expansiveness, shadowing, and specification properties. One important feature of hyperbolic elementary sets is the existence of the unique measure of maximal entropy, which was proved independently by Bowen \cite{B} and Margulis \cite{M} using distinct techniques. The nature of Bowen's proof is more topological, to some extent independent of direct estimates of rate of expansions provided by hyperbolicity. In fact, Bowen's proof enables construction of a measure of maximal entropy for expansive homeomorphisms on compact metric spaces, satisfying the periodic specification property (and the measure is unique in that case). These assumptions are satisfied, in particular, for (mixing) topologically hyperbolic homeomorphisms, that are homeomorphisms which are expansive and satisfy the shadowing property (see the monograph Aoki and Hiraide \cite{AH} for more information on topological hyperbolicity).

Since the ``Beyond Hyperbolicity'' article of Shub and Smale \cite{SS}, many distinct generalizations of hyperbolicity were considered and many results regarding existence of a measure of maximal entropy for these notions were obtained; for partially hyperbolic diffeomorphisms see for example \cite{BFSV,CFT,CFT2,CPZ,HHTU,KT,RT,U,UVY}, and
for expansive systems beyond the specification property see \cite{CT,CT2}. On the other hand, examples of systems without measures of maximal entropy were discussed in \cite{Bu} and \cite{M}. In the case of non-expansive systems, when only shadowing is assumed, still some strong results about entropy can be obtained \cite{MO,LO}, however distinguishing between cases with or without measure of maximal entropy is a challenging problem in general. It is possible to find ergodic measures with entropy approaching any number $c<h(f)$ and converging in weak* topology, but  unfortunately, this does not transfer onto existence of measures of maximal entropy, neither is sufficient to force existence of periodic points. Therefore it seems that some additional properties are needed to ensure existence of such objects, yet weaker than expansivity to go beyond the class of hyperbolic maps.

Recently, a program \emph{beyond topological hyperbolicity} was initiated \cite{ACCV} considering many generalizations of expansiveness and shadowing properties in the literature. In particular, continuum-wise expansiveness was introduced by Kato in \cite{K1} and \cite{K2}, and the study of cw-expansive homeomorphisms satisfying the shadowing property is proving to be a fruitful research direction, containing distinct levels of generalized expansiveness such as N-expansiveness \cite{CC,CC2}, countably expansiveness
\cite{ACCV,ACCV2}, and the more general cw-expansiveness \cite{AAV,ACS,ACCV3,CR}. These systems seem to be a good direction towards reaching properties such as existence of measures of maximal entropy.
An essential complication that arises in this approach is that cw-expansive homeomorphisms satisfying the shadowing property do not mimic well-known results (and very useful) for the hyperbolic dynamics, such as the Spectral Decomposition Theorem
(see \cite{CC} for an example with an infinite number of distinct chain-recurrent classes). This problem seems to be avoided with the introduction of continuum-wise hyperbolicity by Artigue, Carvalho, Cordeiro and Vieitez in \cite{ACCV3}: 
 
\begin{definition}[cw-hyperbolicity]\label{cwhyp}
A homeomorphism $f\colon X\to X$ of a compact metric space $(X,d)$ is \emph{continuum-wise expansive} if there exists $c>0$ such that $W^u_c(x)\cap W^s_c(x)$ is totally disconnected for every $x\in X$, where $W^s_c(x)$ and $W^u_c(x)$ denote the classical $c$-stable/unstable sets of $x$. The number $c>0$ is called a cw-expansive constant of $f$. We say that $f$ satisfies the \emph{cw-local-product-structure} if for each $\eps>0$ there exists $\delta>0$ such that $$C^s_\eps(x)\cap C^u_\eps(y)\neq\emptyset \,\,\,\,\,\, \text{whenever} \,\,\,\,\,\, d(x,y)<\delta,$$ where $C^s_\eps(x)$ and $C^u_\eps(x)$ denote the connected component of $x$ on the sets $W^s_\eps(x)$ and $W^u_\eps(x)$, respectively. The cw-expansive homeomorphisms satisfying the cw-local-product-structure are called \emph{continuum-wise hyperbolic}.
\end{definition}

Beyond hyperbolic case, the sets $W^s_\eps(x)$ and $C^s_\eps(x)$ (similarly $W^u_\eps(x)$ and $C^u_\eps(x)$) can differ a lot, however in the case of cw-hyperbolic homeomorphisms it is still possible to recover some properties. While, on surfaces, $C^s_\eps(x)$ is an arc (assuming cw$_F$-hyperbolicity \cite{ACS}), $W^s_\eps(x)$ can be non-locally-connected and contain a Cantor set of distinct arcs. However, a kind of hyperbolic-like behavior is present inside local stable/unstable continua, which can be obtained for example by considering a cw-metric as in \cite{ACCV3} (see also Theorem \ref{teoCwHyp}). For expansive homeomorphisms, the local-product-structure is defined assuming an intersection between $W^s_\eps(x)$ and $W^u_\eps(y)$ and it is an important tool to prove shadowing, periodic shadowing, and periodic specification (assuming mixing) properties, and in this case sets $W^s_\eps(x)$, $W^u_\eps(y)$ are continua (see the proof of Lemma 1.3 in \cite{Hi}). 
By analogy, the cw-local-product-structure is defined assuming an intersection between $C^s_\eps(x)$ and $C^u_\eps(y)$, which is sufficient to ensure that cw-hyperbolicity implies the shadowing property (see \cite{ACCV3}*{Theorem 2.5}). Unfortunately, the cw-local-product-structure is a much weaker tool than the standard local product structure existing in hyperbolic dynamics and similarities exist mainly only on the conceptual level.

The question about density of periodic points in the non-wandering set for cw-hyperbolic homeomorphisms was discussed in \cite{CR}. It is solved in \cite{CR}*{Theorem 1.3} with the additional hypothesis of jointly continuous stable/unstable holonomies (see Definition \ref{cw-cont}). 

In this article, we prove (a stronger result) that cw-hyperbolic homeomorphisms with jointly continuous stable/unstable holonomies satisfy the periodic shadowing property and consequently, if they are topologically mixing, the periodic specification property (this is done in Section 2). It ensures proper construction of periodic points, allowing us to precise control of their orbits. Among other important consequences, it ensures that ergodic measures supported on periodic orbits are dense in the simplex of invariant measures.

Then we turn to the problem of the existence of measure of maximal entropy, with main difficulty coming from the lack of expansivity.
We discuss difficulties to adapt Bowen's techniques to obtain a measure of maximal entropy for cw-hyperbolic homeomorphisms, exhibit the measure of maximal entropy for Walter's pseudo-Anosov diffeomorphism of $\mathbb{S}^2$, and prove it can be obtained, as in the expansive case, as the weak* limit of an average of Dirac measures on periodic orbits (this is done in Section 3). This result is used in Section 4 where we exhibit the measure of maximal entropy for a homeomorphism on the Sierpi\'nski Carpet defined in \cite{BO}. This example is topologically mixing, has a dense set of periodic points, but does not satisfy the specification property, so Bowen's techniques also do not directly apply. The last section contains further remarks and open questions.

\section{Periodic shadowing and periodic specification}

In this section, we prove periodic shadowing and periodic specification for cw-hyperbolic homeomorphisms with jointly continuous stable/unstable holonomies. We begin with all necessary definitions.

\begin{definition}[Shadowing]
We say that a homeomorphism $f:X\rightarrow X$ satisfies the \emph{shadowing property} if given any $\varepsilon>0$ there is $\delta>0$ such that for each sequence $(x_n)_{n\in\mathbb{Z}}\subset X$ satisfying
$$d(f(x_n),x_{n+1})<\delta \,\,\,\,\,\, \text{for every} \,\,\,\,\,\, n\in\mathbb{Z}$$ there is $y\in X$ such that
$$d(f^n(y),x_n)<\varepsilon \,\,\,\,\,\, \text{for every} \,\,\,\,\,\, n\in\mathbb{Z}.$$
In this case, we say that $(x_n)_{n\in\mathbb{Z}}$ is a $\delta-$pseudo orbit of $f$ and that $(x_n)_{n\in\mathbb{Z}}$ is
$\varepsilon-$shadowed by $y$.
\end{definition}

For more information about shadowing we suggest the monographs of Palmer \cite{Pa} and Pilyugin \cite{Pi}. An alternative (classical) way to define shadowing, as in the above definition, is to shadow finite $\delta$-pseudo orbits $(x_n)_{n=0}^{k}$ independent of their \emph{length} $k$. By compactness of $X$, these approaches are equivalent.
The periodic shadowing property, that we now define, shadows finite pseudo-orbits that are periodic with a periodic point. 

\begin{definition}\label{periodicshadowing}[Periodic shadowing]
We say that a homeomorphism $f:X\rightarrow X$ satisfies the \emph{periodic shadowing property} if given $\varepsilon>0$ there is $\delta>0$ such that for each $\delta$-pseudo-orbit $(x_n)_{n=0}^{k}\subset X$ satisfying $x_k=x_0$, there is a periodic point $y\in X$ satisfying $f^{km}(y)=y$ for some $m\geq 1$ that $\eps$-shadows 
$(z_n)_{n=0}^{km}$ defined by $z_{jk+i}=x_i$ for every $i\in\{0,\ldots,k\}$ and $j\in\{0,\ldots,m-1\}$.
We call $(x_n)_{n=0}^{k}$ a periodic $\delta-$pseudo orbit of $f$. 
\end{definition}

It is known that adding machines are minimal homeomorphisms on Cantor spaces satisfying the shadowing property (see \cite{MY} and \cite{MO}). Thus, they are examples with the shadowing property but without the periodic shadowing property. Understanding the relation between shadowing and periodic shadowing is still an open problem even on low-dimensional spaces. It is proved in \cite{KP} that surface homeomorphisms with the shadowing property have periodic points in each $\eps$-transitive class (this does not necessarily imply density of periodic points in non-wandering set nor the periodic shadowing property). In the case of cw-hyperbolic homeomorphisms, there can even exist a Cantor set of distinct points shadowing a given pseudo-orbit \cite{PV}, and it may also happen that a periodic pseudo-orbit is shadowed by non-periodic points. 

The main result of this section is Theorem \ref{thm:tracing} where we prove that cw-hyperbolic homeomorphisms with jointly continuous stable/unstable holonomies satisfy the periodic shadowing property. Before stating and proving this result, we discuss distinct regularities for stable/unstable holonomies of cw-hyperbolic homeomorphisms, as is done in \cite{CR}.

\begin{definition}[Local stable and local unstable holonomies]\label{holonomies}
Let $c>0$ be a cw-expansive constant of $f$, $\eps=\frac{c}{2}$, and choose $\delta'\in(0,\eps)$, given by the $cw$-local-product-structure, such that 
$$C^u_\eps(x)\cap C^s_\eps(y)\neq\emptyset \,\,\,\,\,\, \text{whenever} \,\,\,\,\,\, d(x,y)<\delta'.$$ 
Thus, if $\delta=\frac{\delta'}{2}$, then
for each pair $(x,y)\in X\times X$, with $d(x,y)<\delta$, we can define $\pi^s_{x,y}\colon C^s_{\delta}(x)\to \mathcal{K}(C^s_{\eps}(y))$ by 
$$\pi^s_{x,y}(z)=C_{\eps}^u(z)\cap C^s_{\eps}(y),$$ where $\mathcal{K}(C^s_{\eps}(y))$ denotes the set of all compact subsets of $C^s_{\eps}(y)$. Note that $\pi^s_{x,y}(z)\neq\emptyset$ since $z\in C^s_{\delta}(x)$ implies 
$$d(z,y)\leq d(z,x)+d(x,y)\leq \delta+\delta=\delta',$$ which, in turn, implies
$$C^u_\eps(z)\cap C^s_\eps(y)\neq\emptyset.$$ The unstable holonomy map $\pi^u_{x,y}\colon C^u_{\delta}(x)\to \mathcal{K}(C^u_{\eps}(y))$ is defined similarly by
$$\pi^u_{x,y}(z)=C_{\eps}^s(z)\cap C^u_{\eps}(y).$$
\end{definition}

These holonomies were described using only the cw-local-product-structure. To discuss their regularities, we need to use the cw-metric. Let $\mathcal{C}(X)$ denote the space of all subcontinua of $X$ and define
$$E=\{(p,q,C): C\in \C(X),\, p,q\in C\}.$$
For $p,q\in C$ denote $C_{(p,q)}=(p,q,C)$.
The notation $C_{(p,q)}$ implies that $p,q\in C$ and that $C\in\C(X)$.
Define
$$f(C_{(p,q)})=f(C)_{(f(p),f(q))}.$$
and consider the sets
\[
\C^s_\eps(X)=\{C\in\C(X);\diam(f^n(C))\leq\eps\, \text{ for every }\, n\geq 0\} \,\,\,\,\,\, \text{and}
\]
\[
\C^u_\eps(X)=\{C\in\C(X);\diam(f^{-n}(C))\leq\eps\, \text{ for every }\, n\geq 0\},
\]
where $\diam(A)$ denotes the diameter of the set $A$. These sets contain exactly the $\eps$-stable and $\eps$-unstable continua of $f$, respectively.

\begin{theorem}[Hyperbolic $cw$-metric-\cite{ACCV3}]
\label{teoCwHyp}
If $f\colon X\to X$ is a cw-expansive homeomorphism of a compact metric space $X$, then there is a function $D\colon E\to\R$ satisfying the following conditions.
\begin{enumerate}
\item Metric properties:
\vspace{+0.2cm}
\begin{enumerate}
 \item $D(C_{(p,q)})\geq 0$ with equality if, and only if, $C$ is a singleton,\vspace{+0.1cm}
 \item $D(C_{(p,q)})=D(C_{(q,p)})$,\vspace{+0.1cm}
 \item $D([A\cup B]_{(a,c)})\leq D(A_{(a,b)})+D(B_{(b,c)})$, $a\in A, b\in A\cap B, c\in B$.
\end{enumerate}
\vspace{+0.2cm}
\item Hyperbolicity: there exist constants $\lambda\in(0,1)$ and $\varepsilon>0$ satisfying
\vspace{+0.2cm}
	\begin{enumerate}
	\item if $C\in\C^s_\eps(X)$ then $D(f^n(C_{(p,q)}))\leq 4\lambda^nD(C_{(p,q)})$ for every $n\geq 0$,\vspace{+0.1cm}
	\item if $C\in\C^u_\eps(X)$ then $D(f^{-n}(C_{(p,q)}))\leq 4\lambda^nD(C_{(p,q)})$ for every $n\geq 0$.
	\end{enumerate}
\vspace{+0.2cm}
\item Compatibility: for each $\delta>0$ there is $\gamma>0$ such that
\vspace{+0.2cm}
\begin{enumerate}
\item if $\diam(C)<\gamma$, then $D(C_{(p,q)})<\delta$\,\, for every $p,q\in C$,\vspace{+0.1cm}
\item if there exist $p,q\in C$ such that $D(C_{(p,q)})<\gamma$, then $\diam(C)<\delta$.
\end{enumerate}
\end{enumerate}
\end{theorem}

In what follows, $\eps$ and $\delta$ are chosen as in Definition \ref{holonomies}.

\begin{definition}\label{cw-cont}[Jointly continuous stable/unstable holonomies]. We say that $f$ has \emph{jointly continuous stable/unstable holonomies} if for each $\eta>0$ there exists $\ga>0$ satisfying:
if $d(x,y)<\delta$, $C$ is a subcontinuum of $C^s_{\eps}(x)$, $p,q\in C$, $D(C_{(p,q)})\leq \ga$, $p^*\in \pi^u_{x,y}(p)$, $C'$ is a subcontinuum of $C^u_{\eps}(p)$ containing $p$ and $p^*$, and $D(C'_{(p,p^*)})\leq\ga$, then there exist subcontinua $C^*$ of $C^s_{\eps}(y)$ containing $p^*$ and $C^{**}$ of $C^u_{\eps}(q)$ containing $q$, and there exists $q^*\in C^{*}\cap C^{**}$ such that 
$$|D(C^*_{(p^*,q^*)})|\leq \eta \,\,\,\,\,\, \text{and} \,\,\,\,\,\, |D(C^{**}_{(p,p^*)})|\leq \eta.$$
\end{definition}

The idea is that a rectangle formed by local stable/unstable continua has small sides provided that we know that at least one stable and one unstable of its sides are sufficiently small (see Figures 1 and 2 in \cite{CR}). One of the main steps in the proof of density of periodic points is that the cw-metric can be adapted to be a self-similar hyperbolic cw-metric \cite{CR}*{Theorem 3.2} (following techniques of Artigue in \cite{Ar}) and, in this case, the local stable/unstable holonomies are jointly pseudo-isometric \cite{CR}*{Theorem 3.6}, as in the following definition.

\begin{definition}[Jointly pseudo-isometric stable/unstable holonomies]\label{pseudo}
We say that $f$ has \emph{jointly pseudo-isometric stable/unstable holonomies} if for each $\eta>0$ there exists $\gamma>0$ satisfying:
if $d(x,y)<\delta$, $C$ is a subcontinuum of $C^s_{\eps}(x)$, $p,q\in C$, $D(C_{(p,q)})\leq \gamma$, $p^*\in \pi^s_{x,y}(p)$, $C'$ is a subcontinuum of $C^u_{\eps}(q)$ containing $p$ and $p^*$, and $D(C'_{(p,p^*)})\leq\gamma$, then there exist subcontinua $C^*$ of $C^s_{\eps}(y)$ containing $p^*$ and $C^{**}$ of $C^u_{\eps}(q)$ containing $q$, and there exists $q^*\in C^{*}\cap C^{**}$ such that 
$$\left|\frac{D(C^*_{(p^*,q^*)})}{D(C_{(p,q)})}-1\right|\leq \eta \,\,\,\,\,\, \text{and} \,\,\,\,\,\, \left|\frac{D(C^{**}_{(p,p^*)})}{D(C'_{(q,q^*)})}-1\right|\leq \eta.$$
In the above case we say that $\pi^s_{x,y}$ is a \emph{pseudo-isometric local stable holonomy} and that $\pi^u_{x,y}$ is a \emph{pseudo-isometric local unstable holonomy}.
\end{definition}


Thus, in the proof of the periodic shadowing property, we assume that $D$ is a self-similar hyperbolic cw-metric and that local stable/unstable holonomies are jointly pseudo-isometric.

\begin{theorem}\label{thm:tracing}
If a cw-hyperbolic homeomorphism $f\colon X\to X$ has jointly continuous stable/unstable holonomies, then for each $\alpha>0$ there exist $\delta>0$ and $r>0$ such that any periodic $\delta$-pseudo-orbit $(y_i)_{i=0}^{n}$ with $n\geq r$ is $\alpha$-shadowed by a point $p$ satisfying $f^n(p)=p$.
In particular, $f$ satisfies the periodic shadowing property.
\end{theorem}

\begin{proof}
Before providing detailed proof, let us sketch its idea first. Given $\alpha>0$ we use cw-hyperbolicity to find $r\in(0,\alpha)$ such that stable/unstable continua shrink sufficiently fast forward/backwards under iterates of $f^r$. Then using stable/unstable holonomies, for a given periodic $\delta$-pseudo-orbit $y_0^{(0)},\ldots, y_n^{(0)}$ for $f^r$ we construct a sequence of continua $(H^{(k)}_i)_{k\in\N}$ such that their union $\overline{\left(\bigcup_{k=1}^{+\infty}{H^{(k)}_i}\right)}$ is connected, has diameter at most $\alpha$ and contains $y_i^{(0)}$ as well as $f^{ri}(q_0)$ for some periodic point $q_0$. Periodicity of $q_0$ is a consequence of construction, because $f^{ri}(q_0)\in \overline{\left(\bigcup_{k=j}^{+\infty}{H^{(k)}_i}\right)}$ for any $j\in\N$, diameters of these continua tend to $0$ with $j\to\infty$ and they are constructed in such a way that they intersect for $i=0$ and $n$. In this way, we get that $y_0^{(0)},\ldots, y_n^{(0)}$ is $\alpha$-traced by $q_0$ for $f^r$. Then the argument is extended, showing that in fact $q_0$ is $2\alpha$-tracing
the sequence $y_0^{(0)},\ldots, f^{r-1}(y_0^{(0)}), y_1^{(0)},\ldots, y_n^{(0)},\ldots, f^{r-1}(y_n^{(0)})$ for $f$.
The argument is completed by the observation, that if $x_0,\ldots,x_{r-1}$ is a $\delta$-pseudo orbit whit sufficiently small $\delta$ then $x_0$ can be used as a point $\alpha$-tracing it.
Now, let us present the details.

For each $\alpha>0$, let $\kappa\in(0,\frac{\alpha}{2})$ be given by the compatibility between $D$ and $\diam$, such that
$$\diam(C)\leq\alpha \,\,\,\,\,\, \text{whenever} \,\,\,\,\,\, D(C_{(x,y)})\leq\kappa$$ for some $x,y\in C$.
Let $c\in(0,\frac{\kappa}{2})$ be a cw-expansive constant of $f$ and choose $\eps\in(0,\frac{c}{2})$, given by the compatibility of $D$ and $\diam$, that is
$$D(C_{(x,y)})\leq c \,\,\,\,\,\, \text{whenever} \,\,\,\,\,\, \diam(C)\leq2\eps.$$
Choose $\delta'\in(0,\eps)$, given by the $cw$-local-product-structure, such that 
$$C^u_\eps(x)\cap C^s_\eps(y)\neq\emptyset \,\,\,\,\,\, \text{whenever} \,\,\,\,\,\, d(x,y)<\delta',$$ 
and let $\delta=\frac{\delta'}{2}$. Thus, the holonomy maps $$\pi^s_{x,y}\colon C^s_{\delta}(x)\to \mathcal{K}(C^s_{\eps}(y)) \,\,\,\,\,\, \text{and} \,\,\,\,\,\, \pi^u_{x,y}\colon C^u_{\delta}(x)\to \mathcal{K}(C^u_{\eps}(y))$$ are well defined when $d(x,y)<\delta$ and are pseudo-isometric, if we work with sufficiently small continua. To make it precise,
choose $\gamma\in(0,\frac{\delta}{2})$, as in Definition \ref{pseudo} of jointly pseudo-isometric stable/unstable holonomies for $\eta=\delta$, and let $\beta\in(0,\min\{\frac{\gamma}{2},\frac{\kappa}{2c}-1\})$ be given by the compatibility between $D$ and $\diam$, such that
$$\diam(C)\leq\gamma \,\,\,\,\,\, \text{whenever} \,\,\,\,\,\, D(C_{(x,y)})\leq2\beta$$ for some $x,y\in C$.
Let $r\in\N$ be such that $4\lambda^{r}c\leq\beta$ and 
$$\sum_{n=1}^{\infty}[(1+\delta)^24]^n\lambda^{nr}\leq\beta.$$
This last inequality is ensured by Lemma 2.8 in \cite{CR}. Observe that if we increase $r$, the above inequality holds as well.
Let $y_0^{(0)},\ldots, y_{n-1}^{(0)}$ be points such that $d(f^r(y_i^{(0)}),y_{i+1}^{(0)})<\delta$ for $i=0,\ldots,n-1$, where we put $ y_{n}^{(0)}= y_{0}^{(0)}$ and $y_{-1}^{(0)}= y_{n-1}^{(0)}$. In other words, we fix an arbitrary periodic $\delta$-pseudo orbit for $f^r$.
The cw-local-product-structure ensures that for each $i\in\{0,\dots,n-1\}$ there exists $$z^{(0)}_i\in C^u_\eps(f^r(y_{i-1}^{(0)}))\cap C^s_\eps(y_i^{(0)}).$$ We let $ z_{n}^{(0)}= z_{0}^{(0)}$ and $z_{-1}^{(0)}= z_{n-1}^{(0)}$.
The hyperbolicity of $D$ and the choice of the constants ensure that
$$D(f^r (C^s_\eps (y_i^{(0)}))_{(f^r(y_i^{(0)}),f^r(z_i^{(0)}))})\leq4\lambda^rc\leq\beta.$$
It follows that $\diam(f^r (C^s_\eps (y_i^{(0)}))<\delta$ and consequently $f^r (C^s_\eps (y_i^{(0)}))\subset C^s_\delta(f^r(y_i^{(0)})) $. 
Let $C_{i+1}^{(0)}=f^r(C^s_\eps (y_{i}^{(0)}))$, $p=f^r(y_{i}^{(0)})$, $q=f^r(z_{i}^{(0)})$, and 
$$p^*=z_{i+1}^{(0)}\in C^u_\eps(f^r(y_{i}^{(0)}))\cap C^s_\eps(y_{i+1}^{(0)})=\pi^s_{f^r(y_{i}^{(0)}),y_{i+1}^{(0)}}(f^r(y_{i}^{(0)})).$$
Since $f$ has pseudo-isometric local stable holonomies, there exist a continuum $C^{(1)}_{i+1}\subset C^s_\eps(y_{i+1}^{(0)})$ containing $p^*$ and 
$$q^*\in \pi^s_{f^r(y_{i}^{(0)}),y_{i+1}^{(0)}}(f^r(z_{i}^{(0)}))\cap C^{(1)}_{i+1}=C^u_\eps(f^r(z_{i}^{(0)}))\cap C^s_\eps(y_{i+1}^{(0)})\cap C^{(1)}_{i+1}$$ 
such that $$D((C^{(1)}_{i+1})_{(p^*,q^*)})\leq (1+\delta)D((C_{i+1}^{(0)})_{(p,q)})\leq (1+\delta)\beta<\gamma.$$ 
Let $y_i^{(1)}=f^{-r}(q^*)$ and note that $d(z_{i+1}^{(0)},f^{r}(y_{i}^{(1)}))<\delta/2$ by the choice of $\beta$. Furthermore, 
$$
y_i^{(1)}\in f^{-r}(C^u_\eps(f^r(z_{i}^{(0)}))):=\hat{C}^{(1)}_{i}\subset C^u_\delta(z_{i}^{(0)})
$$ 
and the hyperbolicity of $D$ and the choice of $\delta$ and $r$ ensure that
$$
D((\hat{C}^{(1)}_{i})_{(y_i^{(1)},z_i^{(0)})})<4\lambda^rc, \,\,\,\,\,\, \diam(\hat{C}^{(1)}_{i})<\delta/2, \,\,\,\,\,\, \text{and} \,\,\,\,\,\, d(y_i^{(1)},z_i^{(0)})<\delta/2.
$$
In particular, the above argument applied for every $i\in\{0,\dots,n-1\}$ ensures that
\begin{eqnarray*}
d(f^r(y_{i}^{(1)}),y_{i+1}^{(1)})&\leq& d(f^r(y_{i}^{(1)}),z_{i+1}^{(0)}) + d(z_{i+1}^{(0)},y_{i+1}^{(1)})\\
&\leq&\frac{\delta}{2} + \frac{\delta}{2} = \delta,
\end{eqnarray*}
where we let $y_{n}^{(1)}= y_{0}^{(1)}$ and $y_{-1}^{(1)}= y_{n-1}^{(1)}$.
Thus, 
$$f^r(y_{i-1}^{(1)}),z_i^{(0)}\in C^{(1)}_{i}, \,\,\,\,\,\, z_i^{(0)},y_i^{(1)}\in\hat{C}^{(1)}_{i}, \,\,\,\,\,\, z_i^{(0)}\in C^{(1)}_{i}\cap\hat{C}^{(1)}_{i},$$ 
$$D((C^{(1)}_{i})_{(f^r(y_{i-1}^{(1)}),z_i^{(0)})})\leq\gamma, \,\,\,\,\,\, \text{and} \,\,\,\,\,\,  D((\hat{C}^{(1)}_{i})_{(z_i^{(0)},y_i^{(1)})})\leq\gamma.$$
Letting $p=z_i^{(0)}$, $q=f^r(y^{(1)}_{i-1})$, and $p^*=y^{(1)}_i$, the pseudo-isometric joint stable/unstable holonomies ensure the existence of continua $\tilde{C}^{(1)}_{i}\subset C^s_\eps (y_{i}^{(1)})$ containing $p^*$ and $\bar{C}^{(1)}_{i}\subset C^u_\eps(f^r(y^{(1)}_{i-1}))$ containing $q$, and a point $z^{(1)}_i:=q^*\in \tilde{C}^{(1)}_{i}\cap \bar{C}^{(1)}_{i}$ such that
$$D( (\tilde{C}^{(1)}_{i})_{(p^*,q^*)})\leq(1+\delta)D({C}^{(1)}_{i})_{(p,q)}\leq(1+\delta)^24\lambda^rc$$ 
and
$$D( (\bar{C}^{(1)}_{i})_{(q,q^*)})\leq(1+\delta) (\hat{C}^{(1)}_{i})_{(p,p^*)}\leq(1+\delta)4\lambda^rc.$$
To end this step, let $F^{(1)}_i={C}^{(1)}_{i}\cup \hat{C}^{(1)}_{i}$ and $H^{(1)}_i={\hat{C}}^{(1)}_{i}\cup \tilde{C}^{(0)}_{i}$ for every $i\in\{0,\dots,n-1\}$, where $\tilde{C}^{(0)}_{i}:=C^s_\eps(y_{i}^{(0)})$.
Repeating the previous argument for the points $y^{(1)}_0,\dots, y^{(1)}_{n-1}$ and $z^{(1)}_0,\dots, z^{(1)}_{n-1}$, we obtain for each $i\in\{0,\dots,n-1\}$ a point $y^{(2)}_i$ and continua $C^{(2)}_i$ and $\hat{C}^{(2)}_i$
such that
$$f^r(y_{i-1}^{(2)}),z_i^{(1)}\in C^{(2)}_{i}, \,\,\,\,\,\, z_i^{(1)},y_i^{(2)}\in\hat{C}^{(2)}_{i}, \,\,\,\,\,\, z_i^{(1)}\in C^{(2)}_{i}\cap\hat{C}^{(2)}_{i},$$ 
$$D((C^{(2)}_{i})_{(f^r(y_{i-1}^{(2)}),z_i^{(1)})})\leq(1+\delta)D(f^r(\tilde{C}^{(1)}_{i-1})_{(y^{(1)}_{i-1},z^{(1)}_{i-1})})\leq(1+\delta)^34^2\lambda^{2r}c,$$
$$\text{and} \,\,\,\,\,\, D((\hat{C}^{(2)}_{i})_{(y_i^{(2)},z_i^{(1)})})\leq4\lambda^rD( (\bar{C}^{(1)}_{i+1})_{(f^r(y^{(1)}_{i}),z_{i+1}^{(1)})}) \leq(1+\delta)4^2\lambda^{2r}c$$
and also obtain continua $\tilde{C}^{(2)}_{i}\subset C^s_\eps (y_{i}^{(2)})$ containing $y^{(2)}_i$ and $\bar{C}^{(2)}_{i}\subset C^u_\eps(f^r(y^{(2)}_{i-1}))$ containing $f^r(y^{(2)}_{i-1})$, and a point $z^{(2)}_i\in \tilde{C}^{(2)}_{i}\cap \bar{C}^{(2)}_{i}$ such that
$$ 
D((\tilde{C}^{(2)}_{i})_{(y_{i}^{(2)},z_i^{(2)})})\leq(1+\delta)^44^2\lambda^{2r}c \,\,\,\,\,\, \text{and}
$$
$$
D((\bar{C}^{(2)}_{i})_{(f^r(y^{(2)}_{i-1}),z_i^{(2)})})\leq(1+\delta)^24^2\lambda^{2r}c.
$$
As in the previous step, let $F^{(2)}_i={C}^{(2)}_{i}\cup \hat{C}^{(2)}_{i}$ and $H^{(2)}_i={\hat{C}}^{(2)}_{i}\cup \tilde{C}^{(1)}_{i}$ for every $i\in\{0,\dots,n-1\}$.
An induction process ensures that for each $k\in\N$ and $i\in\{0,\dots,n-1\}$ there exist a point $y^{(k)}_i$ and continua $F_ i^{(k)}={C}^{(k)}_{i}\cup \hat{C}^{(k)}_{i}$ containing $y^{(k)}_i$ and $f^r(y^{(k)}_{i-1})$ satisfying
$$
D(({F}^{(k)}_{i})_{(y^{(k)}_i, f^r(y^{(k)}_{i-1}))})\leq2((1+\delta)^24\lambda^r)^kc.
$$
By the construction, we have $y_{n}^{(k)}= y_{0}^{(k)}$ and $y_{-1}^{(k)}= y_{n-1}^{(k)}$.
Since $$D(({F}^{(k)}_{i})_{(y^{(k)}_i, f^r(y^{(k)}_{i-1}))})\to0 \,\,\,\,\,\, \text{when} \,\,\,\,\,\, k\to\infty$$ 
and $D$ is compatible with $\diam$, it follows that both $y^{(k)}_i$ and $f^r(y^{(k)}_{i-1})$ converge to a point $q_i$ when $k\to \infty$. Also,
$$
f^r(q_0)=\lim_{k\to \infty}f^r(y^{(k)}_0)=\lim_{k\to \infty}y^{(k)}_1=q_1
$$
and consecutively $$f^{rj}(q_0)=\lim_{k\to \infty}y^{(k)}_j=q_j \,\,\,\,\,\, \text{for every} \,\,\,\,\,\, j\in\{0,\dots,n\}.$$ In particular,
$$
f^{rn}(q_0)=\lim_{k\to \infty}y^{(k)}_n=\lim_{k\to \infty}y^{(k)}_0=q_0,
$$
that is, $q_0$ is a periodic point. The induction process also ensures the existence of continua $H^{(k)}_i=\tilde{C}_i^{(k-1)}\cup \hat{C}_i^{(k)}$
containing $y^{(k-1)}_i$ and $y^{(k)}_i$ such that
$$
D({H^{(k)}_i}_{(y^{(k-1)}_i, y^{(k)}_i)}))\leq2((1+\delta)^24\lambda^{r})^{k-1}c
$$
The choices of $r$ and $\beta$ ensure that
$$
D\left(\overline{\left(\bigcup_{k=1}^{+\infty}{H^{(k)}_i}\right)}_{(y^{(0)}_i,f^{ri}(q_0))}\right)\leq2c(1+\beta)<\kappa
$$
and, hence,
$$
d(y^{(0)}_i,f^{ri}(q_0))<\alpha \,\,\,\,\,\, \text{for every} \,\,\,\,\,\, i\in\{0,\dots,n\},
$$
that is, $q_0$ is a periodic point that $\alpha$-shadows the periodic pseudo-orbit for $f^r$. 
Let $(x_k)_{k=0}^{nr}$ be the periodic $\delta$-pseudo-orbit of $f$ obtained by considering the segments of orbit of $f$ from 0 to $r$ of the points $(y_i^{(0)})_{i=0}^{n-1}$. We prove that $(x_k)_{k=0}^{rn}$ is $2\alpha$-shadowed by $q_0$.
Indeed, if $i\in\{0,\dots,n-1\}$, $j\in\{0,\dots,r\}$, and $k\geq2$, then
\begin{eqnarray*}
D(f^j({H^{(k)}_i})_{(f^j(y^{(k-1)}_i), f^j(y^{(k)}_i))}))&\leq&\lambda^{-j}D({H^{(k)}_i}_{(y^{(k-1)}_i, y^{(k)}_i)}))\\
&\leq&2((1+\delta)^24\lambda^{r})^{k-1}\lambda^{-j}c\\
&\leq&2((1+\delta)^24\lambda^{r})^{k-2}
\end{eqnarray*}
and therefore
$$
D\left(\overline{\left(\bigcup_{k=2}^{+\infty}{f^j(H^{(k)}_i)}\right)}_{(f^j(y^{(1)}_i),f^{ri+j}(q_0))}\right)\leq2c(1+\beta)<\kappa,
$$
which ensures that
$$
d(f^j(y^{(1)}_i),f^{ri+j}(q_0))<\alpha.
$$
Since $H_i^{(1)}=\hat{C}_i^{(1)}\cup\tilde{C}_i^{(0)}$, $\hat{C}_i^{(1)}=f^{-r}(C^u_\eps(f^r(z_{i}^{(0)})))$, and $\tilde{C}_i^{(0)}=C^s_\eps(y_i^{(0)})$, it follows that
$$
d(f^j(y_i^{(0)}),f^j(y_i^{(1)}))\leq \diam(f^j(H_i^{(1)}))\leq 2\eps
$$
which gives 
$$d(f^j(y_i^{(0)}),f^{ir+j}(q_0))<\alpha+2\eps<2\alpha.$$
This proves that $(x_k)_{k=0}^{rn}$ is $2\alpha$-shadowed by $q_0$. To finish the proof, we note that using continuity of $f^r$ we can reduce $\delta$, if necessary, to ensure that any $\delta$-pseudo-orbit of $f$ with length $r$ is $\alpha$-shadowed by its initial point. Thus, we can split any periodic $\delta$-pseudo-orbit $(y_k)_{k=0}^{rn}$ of $f$ into pieces of length $r$ and obtain a periodic-pseudo-orbit of $f^r$ that is $\alpha$-shadowing it. The shadowing orbit obtained as before will $3\alpha$-shadow $(y_k)_{k=0}^{rn}$.
This deals with the case of pseudo-orbits with lengths multiples of $r$. To deal with the general case, we need a small adjustment. Namely, observe that we do not have to consider the sequence of pieces of pseudo-orbits of constant length $r$. Simply, if we have $r\leq r_i< 2r$ and $y_0^{(0)},\ldots, y_{n-1}^{(0)}$ be points such that $d(f^{r_i}(y_i^{(0)}),y_{i+1}^{(0)})<\delta$, then at the end we will get $f^{r_1+\ldots+r_n}(q_0)=q_0$. This will allow us to work with periodic $\delta$-pseudo orbit of suitable length. In the last step of the proof, we decreased $\delta$, but since all $r_i\leq 2r$, this can be done independently of the pseudo-orbit (using only the continuity of the maps $f, f^2,\ldots, f^{2r}$). 

This proves the first statement of the theorem and proves, in particular, that $f$ satisfies the periodic shadowing property since for pseudo-orbits $(y_i)_{i=0}^{n}$ with $n<r$ we can just consider the pseudo-orbit $(z_k)_{k=0}^{nm}$ (as in Definition \ref{periodicshadowing}) where $m\in\N$ is the first positive integer number such that $nm\geq r$, and obtain a periodic point $q=f^{nm}(q)$ shadowing it. This finishes the proof.
\end{proof}

\begin{remark}\label{expansiveandcw}
It is not hard to verify that if $f\colon X\to X$
is an expansive homeomorphism, $(y_j)_{j=0}^{n}$ is a periodic $\delta$-pseudo-orbit, and $q=f^{nr}(q)$ is a periodic point $\alpha$-shadowing the pseudo-orbit $(z_j)_{j=0}^{nr}$ as in Definition \ref{periodicshadowing}, then $f^n(q)=q$ (provided $2\alpha$ is an expansive constant). This means that periodic pseudo-orbits of a certain length $n$ are shadowed by $n$-periodic points, that is, there is a match between the length of the pseudo-orbit and the period of the periodic point shadowing it. In Theorem \ref{thm:tracing}, we were only able to obtain a similar match for pseudo-orbits with sufficiently large length $n\geq r$, where $r$ depends on the constant $\alpha$. This can be an important difference between the periodic shadowing property in the topologically hyperbolic and the cw-hyperbolic cases (see Question \ref{5}).
\end{remark}

The periodic specification property follows directly from topological mixing and the periodic shadowing property.

\begin{definition}[Topological mixing]
A map $f\colon X\to X$ is called \emph{topologically mixing} if for any pair $U,V\subset X$ of non-empty open subsets, there exists $n\in\N$ such that $$f^k(U)\cap V\neq\emptyset \,\,\,\,\,\, \text{for every} \,\,\,\,\,\, k\geq n.$$
\end{definition}

\begin{definition}[Specification]
Let $\tau=\{I_1,\dots,I_m\}$ be a finite collection of disjoint finite subsets of consecutive integers, $I_j=[a_j,b_j]\cap\Z$ for some $a_j,b_j\in\Z$,
with
$$a_1\le b_1 < a_2\le b_2 <\ldots < a_m\le b_m.$$ Let a map $P\colon \bigcup_{j=1}^mI_j\rightarrow X$ be such that for each $I\in \tau$ and $t_1, t_2\in I$
we have $$f^{t_2-t_1}(P(t_1))=P(t_2).$$ We call a pair $(\tau,P)$ a \emph{specification}. We say that the specification $S=(\tau,P)$ is \emph{$L$-spaced}
if $a_{j+1}\geq b_j+L$ for all $j\in\{1,\dots,m-1\}$. Moreover, $S$ is \emph{$\eps$-shadowed} by $y\in X$ if $$d(f^k(y),P(k))<\eps \,\,\,\,\,\,
\textrm{for every} \,\,\, k\in \bigcup_{j=1}^mI_j.$$
We say that a homeomorphism $f\colon X\rightarrow X$ has the \emph{specification property} if for every $\eps>0$ there exists $L\in\N$ such that every
$L$-spaced specification is $\eps$-shadowed.
\end{definition}

\begin{definition}[Periodic specification]
We say that $f$ has the \emph{periodic specification property} if for every $\eps>0$ there exists $L\in\N$
such that every $L$-spaced specification is $\eps$-shadowed by a periodic point $y$ such that $f^{b_m+L}(y)=y$.
\end{definition}

The specification property was introduced by Bowen in \cite{B} and is an important property in the study of topological and statistical properties in dynamical systems. There are other important classes beyond the hyperbolic systems which satisfy the specification property, such as mixing interval maps or, more generally, graph maps, mixing cocyclic shifts, among others (see \cite{Blokh, Blokh2}, more recent \cite{Buzzi,HKO}, or survey paper \cite{KLO}).

\begin{theorem}\label{thm:perspec}
If a topologically mixing and cw-hyperbolic homeomorphism $f\colon X\to X$ has jointly continuous stable/unstable holonomies, then it satisfies the periodic specification property.
\end{theorem}
\begin{proof}
For each $\eps>0$, let $\delta>0$ and $r\in\N$ be provided to it by Theorem~\ref{thm:tracing}.
Mixing ensures that for any $\alpha>0$, there is $M_\alpha\in\N$ such that for any $x,y\in X$ and any $n\geq M_\alpha$ there is a $\alpha$-pseudo-orbit of length $n$ connecting $x$ and $y$. Enlarge $r$ when necessary, so that $r\geq M_\delta$.
For each $r$-spaced specification $\xi = {T^{[a_j,b_j]}(x_j)}^{n-1}_{j=0}$ we can use mixing as explained above to find a periodic $\delta$-pseudo-orbit $(\eta_i)_{i=0}^{b_{n-1}+r}$  such that $$\eta_i=T^i(x_j) \,\,\,\,\,\, \text{for every} \,\,\,\,\,\, i\in [a_j,b_j].$$ 
Since length of $\eta$ is clearly larger than $r$, it is enough to apply Theorem~\ref{thm:tracing} to complete the proof.
\end{proof}

\section{Pseudo-Anosov diffeomorphism of $\mathbb{S}^2$}

Pseudo-Anosov diffeomorphisms are classical examples of expansive systems on surfaces and their study comes back to works of Thurston on classification of surface diffeomorphisms \cite{T}. In these systems, stable/unstable sets form a pair of transversal measured foliations with a finite number of singularities with a number of separatrices bigger than two \cite{Hi2}. Walters considered in \cite{W}*{Example 1, p. 140} an example to show that a factor of an expansive homeomorphism may not be expansive. It contains a pair of stable/unstable foliations with four singularities but is not pseudo-Anosov as in Thurston's definition since it is not expansive and its singularities are in the form of spines and have only one separatrix. This example was further explored in \cite{ACS}, \cite{Ar2}, \cite{ACCV}, \cite{ACCV2}, \cite{ACCV3}, \cite{CR}, \cite{PPV}, \cite{PV}, and now it is understood that the uniform dynamical structure that appears in this example is cw-hyperbolicity. In particular, it is cw$_2$-hyperbolic (see \cite{Ar2}*{Proposition 2.2.1}) and stable/unstable holonomies are jointly continuous, as proved in \cite{CR}*{Theorem 1.4}. We will make use of these facts later in our considerations, enabling us to apply Theorem~\ref{thm:tracing}.

We can describe this example as follows. Let $A$ be a hyperbolic matrix in $SL(2,\mathbb{Z})$ with eigenvalues $\lambda>1$ and $\lambda^{-1}\in(0,1)$. This matrix induces diffeomorphisms $f_A\colon\mathbb{T}^2\to\mathbb{T}^2$ and $g_A\colon\mathbb{S}^2\to\mathbb{S}^2$ where $f_A$ is an Anosov diffeomorphism and $g_A$ is a cw-Anosov diffeomorphism (see \cite{ACCV3}). The quotient map $\pi\colon\mathbb{T}^2\to\mathbb{S}^2$ to define $g_A$ is defined by relating antipodal points $x\sim-x$ in addition to the classical Torus relation on $\left[-\frac{1}{2},\frac{1}{2}\right] \times\left[-\frac{1}{2},\frac{1}{2}\right]$.
For each $n\in\N$, let
$$P_n(f_A)=\{x\in\mathbb{T}^2; \,\,f_A^n(x)=x\} \,\,\,\,\,\, \text{and} \,\,\,\,\,\, P_n^{-}(f_A)=\{x\in\mathbb{T}^2; \,\,f_A^n(x)=-x\}.$$ Points in the first set are called \emph{$n$-periodic} and in the second are called \emph{antipodal $n$-periodic}. Let $\Per_n(f_A)$ and $\Per_n^{-}(f_A)$ denote the cardinality of $P_n(f_A)$ and $P_n^{-}(f_A)$, respectively. These numbers can be calculated as follows.

\begin{lemma}\label{lem:per_and_perminus}
The following holds for every $n\in\N$:
$$\Per_n(f_A)=\lambda^n+\lambda^{-n}-2 \,\,\,\,\,\, \text{and} \,\,\,\,\,\, \Per_n^{-}(f_A)=\lambda^n+\lambda^{-n}+2.$$
\end{lemma}

\begin{proof}
The equation $f_A^n(x)=x$ is equivalent to $(f_A^n-Id)(x)=(0,0)$. The map $f_A^n-Id$ is a well-defined map on $\mathbb{T}^2$ that is induced by the linear map $A^n-Id$ on $\mathbb{R}^2$. The determinant of $A^n-Id$ is given by 
$$(\lambda^n-1)(\lambda^{-n}-1)=2-\lambda^n-\lambda^{-n}$$
and its modulus is exactly the area of the set $(A^n-Id)([0,1]\times[0,1])$ and determines exactly the number of pre-images of $f_A^n-Id$ at any point of $\mathbb{T}^2$. In particular, $(f_A^n-Id)^{-1}(0,0)$ contains exactly $\lambda^n+\lambda^{-n}-2$ points of $\mathbb{T}^2$ and, hence, $\Per_n(f_A)=\lambda^n+\lambda^{-n}-2$. A similar discussion can be used to determine the number of points $x\in\mathbb{T}^2$ satisfying $f_A^n(x)=-x$. The equation $f_A^n(x)=-x$ is equivalent to $(f_A^n+Id)(x)=(0,0)$ and the same reason explained above ensures that the number of solutions of $(f_A^n+Id)(x)=(0,0)$ is given by the modulus of the determinant of $A^n+Id$, that equals $\lambda^n+\lambda^{-n}+2$. This proves the equality $\Per_n^{-}(f_A)=\lambda^n+\lambda^{-n}+2$.
\end{proof}

Now we consider periodic points for $g_A$.
For each $n\in\N$, let 
$$P_n(g_A)=\{x\in\mathbb{S}^2; \,\, g_A^n(x)=x\} \,\,\,\,\,\, \text{and} \,\,\,\,\,\, \Per_n(g_A)=\#P_n(g_A).$$ In the following proposition, we calculate the number $\Per_n(g_A)$.

\begin{proposition}\label{PergA}
The following holds for every $n\in\N$:
\begin{enumerate}
    \item $\pi\colon P_n(f_A)\cup P^{-}_n(f_A)\to P_n(g_A)$ is 2 to 1;
    \item $\Per_n(g_A)=\frac{\Per_n(f_A)}{2}+\frac{\Per_n^{-}(f_A)}{2}=\lambda^n+\lambda^{-n}.$
\end{enumerate}
\end{proposition}
\begin{proof}
First, note that if $f^n_A(x)=x$, then $f_A^{n}(-x)=-x$, $\pi(x)=\pi(-x)$, and $g_A^n(\pi(x))=\pi(x)$. This means that periodic points of $f_A$ exist in antipodal pairs (besides the following four points $(0,0)$, $(\frac{1}{2},0)$, $(0,\frac{1}{2})$, and $(\frac{1}{2},\frac{1}{2})$, which have the same projection 
by $\pi$ as their antipodal points). Also, if $f_A^n(y)=y$, $y\neq x$, and $y\neq-x$, then $\pi(y)\neq\pi(x)$ and $g_A^n(\pi(y))=\pi(y)$, so each pair projects into a distinct periodic point of $g_A$.
Moreover, if $f^n_A(x)=-x$, then $f^n_A(-x)=x$, $f^{2n}_A(x)=x$, and $g^n_A(\pi(x))=\pi(x)$. Thus, antipodal $n$-periodic points of $f_A$ also project to $n$-periodic points of $g_A$.
If $g_A^n(y)=y$ and $x\in\pi^{-1}(y)$, then either $f_A^n(x)=x$ and $f_A^n(-x)=-x$, or $f_A^n(x)=-x$ and $f_A^n(-x)=x$. This means that each periodic point of $g_A$ is a projection of a pair of antipodal points (besides the four points stated above) that are either periodic or antipodal periodic for $f_A$. To conclude the first item and the first equality in the second item, we just have to note that those four points are both periodic and antipodal periodic and, consequently, belong to both $P_n(f_A)$ and $P^{-}_n(f_A)$, and are being counted in both $\Per_n(f_A)$ and $\Per_n^{-}(f_A)$. The second equality follows by using the equalities $\Per_n(f_A)=\lambda^n+\lambda^{-n}-2$ and $\Per_n^{-}(f_A)=\lambda^n+\lambda^{-n}+2$ proved in Lemma \ref{lem:per_and_perminus}.
\end{proof}
As a consequence, we conclude, in the following proposition, that Bowen's formula for calculating the topological entropy as the exponential growth of the number of periodic points also holds for $g_A$. First, we recall Bowen's definition of the topological entropy. 

\begin{definition}[Topological entropy]
Let $f\colon X\to X$ be a continuous map of a compact metric space. Given $n\in\N$ and $\delta>0$, we say that $E\subset X$ is $(n,\delta)$-\emph{separated}
if for each $x,y\in E$, $x\neq y$, there is $k\in \{0,\dots,n-1\}$ such that
$\dist(f^k(x),f^k(y))>\delta$.
Let $F\subset X$ and $s_n(F,\delta)$ denote the maximal cardinality of an $(n,\delta)$-separated subset $E\subset F$.
Since $X$ is compact, $s_n(F,\delta)$ is finite.
Let\[
 h(f,F,\delta)=\limsup_{n\to\infty}\frac 1n\log s_n(F,\delta),
\]
note that $h(f,F,\delta)$ increases as $\delta\to 0$, and define $h(f,F)=\lim_{\delta\to 0}h(f,F,\delta)$. When $F=X$ we omit $F$ in the above notation and let $s_n(\delta)=s_n(F,\delta)$, $h(f,\delta)=h(f,F,\delta)$, and define the topological entropy of $f$ as
$h(f)=h(f,X)$.
\end{definition}

\begin{lemma}\label{formula}
The following equality holds:
$$\lim_{n\to\infty}\frac{\log(\Per_n(g_A))}{n}=\log(\lambda)=h(g_A).$$
\end{lemma}

\begin{proof}
The first equality is a direct consequence of Lemma \ref{PergA}:
$$\lim_{n\to\infty}\frac{\log(\Per_n(g_A))}{n}=\lim_{n\to\infty}\frac{\log(\lambda^n+\lambda^{-n})}{n}=\log(\lambda).$$
In the second equality, we prove that the topological entropy of $g_A$ also equals $\log(\lambda)$. Indeed, $g_A$ being a factor of $f_A$ implies $h(g_A)\leq h(f_A)$, while the quotient map satisfying $\#\pi^{-1}(y)\leq2$ for every $y\in\mathbb{S}^2$ implies $h(f_A,\pi^{-1}(y))=0$ for every $y\in\mathbb{S}^2$ and (1.1) in \cite{LW} ensures that $h(f_A)\leq h(g_A)$. This ensures both equalities and finishes the proof.
\end{proof}

\begin{proposition}\label{+-separated}
Let $\eps\in(0,\frac{1}{4})$ be an expansive constant of $f_A$. Then the sets $P_n(f_A)$ and $P^{-}_n(f_A)$ are $(n,\eps)$-separated for every $n\in\N$.
\end{proposition}
 
\begin{proof}
Expansiveness of $f_A$ ensures that $P_n(f_A)$ is $(n,\eps)$-separated and since antipodal n-periodic points are 2n-periodic, expansiveness of $f_A$ ensures that $P^{-}_n(f_A)$ is $(2n,\eps)$-separated. If $x,y\in P_n^{-}(f_A)$, i.e. $f_A^n(x)=-x$ and $f_A^n(y)=-y$, and $$d(f_A^i(x),f_A^i(y))<\eps \,\,\,\,\,\, \text{for every} \,\,\,\,\,\, i\in\{0,\ldots,n-1\},$$
then
$$
d(f_A^{n+i}(x),f_A^{n+i}(y))=d(f_A^{i}(-x),f_A^{i}(-y))=d(f_A^i(x),f_A^i(y))
$$
for every $i\in\{0,\ldots,n-1\}$ and hence 
$$d(f_A^i(x),f_A^i(y))<\eps \,\,\,\,\,\, \text{for every} \,\,\,\,\,\, i\in \Z.$$ But by expansivity this means that $x=y$.
In other words, any two distinct points $x,y\in P_n^{-}(f_A)$
have to $\eps$-separate in the first $n$ iterations. 
\end{proof}

If $f\colon X\to X$ is an expansive homeomorphism and satisfies the periodic specification property, then there are constants $D,E>0$ such that 
$$De^{h(f)n}\leq\Per_n(f)\leq Ee^{h(f)n}$$
for sufficiently large $n\in\N$ (see \cite[Lemma 22.6]{DGS}). Another consequence of Proposition \ref{PergA} is that similar inequalities hold for $g_A$.

\begin{lemma}
The following holds for sufficiently large $n\in\N$:
$$e^{h(g_A)n}\leq\Per_n(g_A)\leq 2e^{h(g_A)n}.$$
\end{lemma}

\begin{proof}
The result follows from the following equalities:
$$\Per_n(g_A)=\lambda^n+\lambda^{-n}=e^{\log(\lambda)n}+e^{-\log(\lambda)n}=e^{h(g_A)n}+e^{-h(g_A)n},$$
observing that $e^{h(g_A)n}+e^{-h(g_A)n}<2e^{h(g_A)n}$ for suffiently large $n\in\N$.
\end{proof}

In the expansive case, similar inequalities hold changing $Per_n(f)$ by $s_n(\eps)$, that is, for each $\eps\in(0,c)$ there are constants $D_{\eps},E_{\eps}>0$ such that
$$D_{\eps}e^{h(f)n}\leq s_n(\eps)\leq E_{\eps}e^{h(f)n} \,\,\,\,\,\, \text{for every} \,\,\,\,\,\, n\geq1$$
(see \cite[Lemma 22.5]{DGS}). But $g_A$ is not expansive and furthermore has horseshoes inside some dynamical balls with arbitrarily small radius (see \cite{PV} and \cite{ACCV}*{Theorem B}). As a consequence, many periodic points for $g_A$ do not separate and the relation between $\Per_n(g_A)$ and $s_n(\eps)$ is not clear. This leads to problems in the general cw-hyperbolic case. However in our context, we can use the estimatives for $f_A$ and properties of the quotient map $\pi$ to describe the measure of maximal entropy of $g_A$ and prove it can be obtained, as in the expansive case, as the weak* limit of an average of Dirac measures on periodic orbits. Let $(\mu_n)_{n\in\N}$ be the sequence of measures on $\mathbb{S}^2$ defined as follows:
$$
\mu_n=\frac{1}{\Per_n(g_A)}\sum_{p\in P_n(g_A)}\delta_p,
$$
where $\delta_p$ is the Dirac measure at $p$. The following is among the main results of this article.

\begin{theorem}\label{mme}
The sequence $(\mu_n)_{n\in\N}$ converges to a measure $\mu$ in the weak* topology and $\mu$ is the unique measure of maximal entropy for $g_A$.
\end{theorem}

Before proceeding to the proof, we recall Bowen's definition of a homogeneous measure. Its importance relies in the fact that homogeneous measures are measures of maximal entropy (see \cite{DGS}*{Proposition 19.7}). 

\begin{definition}[Homogeneous measure]
Let $f\colon X\to X$ be a continuous map and $\nu$ be a probability measure on $X$. Considering the Bowen's dynamical ball of radius $c$ and time $n$,
$$B_n(x,c)=\{y\in X; \,\,d(f^k(y),f^k(x))\leq c \,\,\, \text{for every} \,\,\, k\in\{0,\dots,n-1\}\},$$
we say that $\nu$ is a \emph{homogeneous measure} if for each $\eps>0$ there are $\delta>0$ and $K>0$ such that for every $x,y$ we have
$$
\nu (B_n(y,\delta))\leq K \nu(B_n(x,\eps)).
$$
\end{definition}

\begin{proof}[Proof of Theorem \ref{mme}]
First, note that $f_A$ has a unique measure of maximal entropy $\eta$ and its push forward measure $\mu=\pi_*\eta$ is a measure of maximal entropy of $g_A$, because $\pi$ is bounded to one which guarantees $h_\mu(g_A)=h_\eta(f_A)$ (e.g. see \cite{LW}) and, hence,  $h(g_A)=h(f_A)=h_\eta(f_A)=h_\mu(g_A)$ (see \cite{DGS} for the definition of metric entropy).

Next let us prove the uniqueness of the measure of maximal entropy of $g_A$. To do so, let us fix any  measure of maximal entropy $\hat\mu$  for $g_A$. Theorem~\ref{thm:perspec} implies that $g_A$ has the periodic specification property, which ensures the existence of a sequence $\hat{\mu}_n$ of measures supported on periodic points $p_n$ such that $\hat\mu$ is the weak* limit of $\hat{\mu}_n$ (see \cite[Proposition~21.8]{DGS}). Let $q_n$ be any periodic point $q_n\in \pi^{-1}(p_n)$ and let $\hat{\eta}_n$ be the measure supported on $q_n$. Then $\hat{\mu}_n=\pi_*(\hat{\eta}_n)$. Without loss of generality we may assume that $(\hat{\eta}_n)_{n\in\N}$ converge in the weak* topology to a measure $\hat{\eta}$. But then, by continuity of the push-forward operator $\pi_*$ we have $\pi_*\hat{\eta}=\hat\mu$ and as a consequence
$$
h(g_A)=h_{\hat\mu}(g_A)\leq h_{\hat{\eta}}(f_A)\leq h(f_A)=h(g_A).
$$
This proves that $\hat{\eta}=\eta$ and consequently that $\hat{\mu}=\mu$, that is, $\mu$ is the unique measure of maximal entropy of $g_A$.

It remains to prove that $\mu$ is the weak* limit of $\mu_n$. A difficulty that arise is how to lift precisely the sequence $(\mu_n)_{n\in\N}$ to the Torus since the sequence $(\eta_n)_{n\in\N}$ defined by
$$
\eta_n=\frac{1}{\Per_n(f_A)}\sum_{p\in P_n(f_A)}\delta_p,
$$
which converge to $\eta$ by Bowen's proof, does not project to $(\mu_n)_{n\in\N}$. Indeed, the antipodal periodic points are also in the pre-image of periodic points of $g_A$. An attempt would be to include the antipodal periodic points in the definition of these measures and consider the sequence $(\hat{\eta}_n)_{n\in\N}$ defined by
$$
\hat{\eta}_n=\frac{1}{\Per_n(f_A)+\Per^{-}_n(f_A)}\sum_{p\in P_n(f_A)\cup P^{-}_n(f_A)}\delta_p.
$$
But $\hat{\eta}_n$ still does not project to $\mu_n$ and the problem relies on the existence of the spines (points with a single pre-image). Thus, we rule out these points as follows: for each $n\in\N$, let 
$$P_n^*(g_A)=\{p\in P_n(g_A): \# \pi^{-1}(p)=2\}, \,\,\,\,\,\, \Per_n^*(g_A)=\# P_n^*(g_A),$$
$$
\text{and} \,\,\,\,\,\, \hat{\mu}_n=\frac{1}{\Per^*_n(g_A)}\sum_{p\in P^*_n(g_A)}\delta_p.
$$
Since there are only four points with a single pre-image, we have 
$$|\Per_n(g_A)-\Per^*_n(g_A)|\leq 4.$$ Thus, $(\mu_{n})_{n\in\N}$ and $(\hat{\mu}_{n})_{n\in\N}$ converge weakly* to exactly the same measure, provided the limit exists. For each $n\in\N$, let 
$$P_n^*(f_A)=\pi^{-1}(P_n^*(g_A)) \,\,\,\,\,\, \text{and} \,\,\,\,\,\, \Per_n^*(f_A)=\#P^*_n(f_A)=2\Per^*_n(g_A),$$ and note that for $n$ sufficiently large we have
$$\frac{\Per_n(f_A)}{2}\leq\Per^*_n(f_A)\leq 3\Per_n(f_A).$$
Now we can lift $(\hat{\mu}_{n})_{n\in\N}$ to a sequence $(\nu_{n})_{n\in\N}$ defined by
$$
\nu_n=\frac{1}{\Per^*_n(f_A)}\sum_{p\in P^*_n(f_A)}\delta_p\label{eq:pern*mun},
$$
obtaining that $\hat{\mu}_n=\pi_*(\nu_n)$ for every $n\in\N$. Indeed, for each $E\subset\mathbb{S}^2$ we have
\begin{eqnarray*}
\pi_{*}(\nu_n)(E)&=&\frac{\#(\pi^{-1}(E)\cap P_n^*(f_A))}{\Per_n^*(f_A)}\\
&=&\frac{2\#(E\cap P_n^*(g_A))}{2\Per_n^*(g_A)}\\
&=&\hat{\mu}_n(E).
\end{eqnarray*}

Let $\nu$ be a weak*-limit of the sequence $(\nu_n)_{n\in\N}$. We will prove that $\nu$ is a homogeneous measure and, consequently, a measure of maximal entropy. This will ensure that $\nu=\eta$, since $\eta$ is the unique measure of maximal entropy of $f_A$, and that $(\nu_n)_{n\in\N}$ converges to $\eta$ in the weak*-topology. The continuity of the push-forward operator $\pi_*$ will then ensure that $(\hat{\mu}_n)_{n\in\N}$ will converge to $\mu$ in the weak*-topology, and consequently that $(\mu_n)_{n\in\N}$ will converge to $\mu$ in the weak*-topology, concluding the proof.

To prove that $\nu$ is homogeneous,
for each $n,r\in\N$ and $x\in\mathbb{T}^2$, we estimate $\nu_r(B_n(x,\eps))$ as follows: let $M$ be the constant given by the periodic specification property of $f_A$ for $\eps$, write $r=n+m+2M$, and note that (as in the proof of \cite{DGS}*{Proposition 22.8}) the periodic specification ensures that
$$\#(B_n(x,\eps)\cap P_r^*(f_A))\geq s_m(3\eps),$$
which, in turn, implies
\begin{eqnarray*}
\nu_r(B_n(x,\eps))&=&\frac{\#(B_n(x,\eps)\cap P_r^*(f_A))}{\Per_r^*(f_A)}\\
&\geq&\frac{s_m(3\eps)}{3\Per_r(f_A)}\\
&\geq&3^{-1}E^{-1}D_{3\eps}e^{h(f_A)m}e^{-h(f_A)r}\\
&=&A_{\eps}e^{-h(f_A)n},
\end{eqnarray*}
where $A_{\eps}=3^{-1}E^{-1}D_{3\eps}e^{-h(f_A)2M}$. Since $\nu$ is the limit of a subsequence $(\nu_r)_{r\in S}$, it follows that
$$\nu(B_n(x,\eps))\geq\limsup_{r\in S}\nu_r(B_n(x,\eps))\geq A_{\eps}e^{-h(f_A)n}.$$

Now assume that $\eps$ was chosen such that $3\eps$ is an expansive constant of $f_A$ and note that since $P_r^*(f_A)\subset P_r(f_A)\cup P^{-}_r(f_A)$ and the sets $P_r(f_A)$ and $P^{-}_r(f_A)$ are $(r,3\eps)$-separated (see Proposition \ref{+-separated}), we have that any $x,z\in B_n(y,3\eps)\cap P_r(f_A)$ (and any $x,z\in B_n(y,3\eps)\cap P_r^-(f_A)$) 
are such that $f_A^n(x)$ and $f_A^n(z)$ are $(r-n,3\eps)$-separated. This implies that
\begin{eqnarray*}\#(B_n(y,3\eps)\cap P_r^*(f_A))
&\leq& \#(B_n(y,3\eps)\cap P_r(f_A))+
\#(B_n(y,3\eps)\cap P^-_r(f_A))\\
&\leq& 2s_{r-n}(3\eps)
\end{eqnarray*}
and consequently,
\begin{eqnarray*}
\nu_r(B_n(y,3\eps))&=&\frac{\#(B_n(y,3\eps)\cap P_r^*(f_A))}{\Per_r^*(f_A)}\\
&\leq&\frac{4s_{r-n}(3\eps)}{\Per_r(f_A)}\\
&\leq&4D^{-1}E_{3\eps}e^{h(f_A)(r-n)}e^{-h(f_A)r}\\
&=&B_{\eps}e^{-h(f_A)n},
\end{eqnarray*}
where $B_{\eps}=4D^{-1}E_{3\eps}$.
For each $n\in\N$ consider the open set $$V_n=\{x\in\mathbb{T}^2; \,\, d(f_A^i(x),f^i_A(y))<2\eps \,\, \text{for every} \,\, i\in\{0,\ldots,n\}\}$$
and note that $B_n(y,\eps)\subset V_n\subset B_n(y,3\eps)$, which ensures that $\nu_r(V_n)\leq B_{\eps}e^{-h(f_A)n}$.
Since $V_n$ is an open set, it follows that
\begin{eqnarray*}
\nu(B_n(y,\eps))\leq\nu(V_n)&\leq& \liminf_{r\in S} \nu_{r}(V_n)\\
&\leq& \liminf_{r\in S} \nu_{r}(B_n(y,3\eps))\leq B_{\eps}e^{-h(f_A)n}.
\end{eqnarray*}
Thus, we proved that for each $\eps>0$ there is $K=B_{\eps}A_{\eps}^{-1}>0$ such that for every $x,y$ we have
$$
\nu (B_n(y,\eps))\leq K \nu(B_n(x,\eps)).
$$
In particular, $\nu$ is a homogeneous measure for $f_A$. This completes the proof.
\end{proof}

\section{Measure of Maximal Entropy on the Sierpi\'nski Carpet}

In this section, we exhibit the measure of maximal entropy for a homeomorphism on the Sierpi\'nski Carpet defined in \cite{BO}. First, we need to briefly recall after \cite{BO} how this map is constructed. We begin considering the pseudo-Anosov diffeomorphism $g_A\colon \mathbb{S}^2\to\mathbb{S}^2$ and follow the construction in \cite{BO} described in what follows. Let $\SC=\{x_1,...,x_4\}$ be the spines of $g_A$ and $\SP$ be a sequence of periodic points of $g_A$ satisfying:
\begin{enumerate}
\item $\SP\cap\SC=\emptyset$,
\item $\SP$ is dense on $\mathbb{S}^2$,
\item $\Per(g_A)\setminus \SP$ is dense on $\mathbb{S}^2$.
\end{enumerate}
Let $\SO=\{O_1,O_2,...\}$ be the sequence of periodic orbits of the points in $\SP$. We assume that these orbits have pairwise distinct periods. Write $S_0=\mathbb{S}^2$, $g_0=g_A$, and $O_1=\{p_1,g_0(p_1),...,g_0^{n_1-1}(p_1)\}$.
Following \cite{BO} we can blow-up each point $g_0^i(p_1)$ into a circle $C_{1,i}$. Precisely, there is a pair $(S_1,\pi_1:S_1\to S_0)$ such that $\pi_1^{-1}(g_0^{i}(p_1))$ is a circle $C_{1,i}$ of radial directions centered at $g_0^{i}(p_1)$ and the collapsing map $\pi_1$, when restricted to  
$S_1\setminus (C_{1,0}\cup\dots\cup C_{1,n_1-1})$,
is a homeomorphism onto $\mathbb{S}^2\setminus O_1$. Moreover, the map $g_0$ induces a map $g_1:S_1\to S_1$, so that:
\begin{itemize}
    \item $g_0\circ \pi_1=\pi_1\circ g_1$ 
    \item $g_1$ maps $C_{1,i}$ homeomorphically into $C_{1,i+1}$ for every $i\in\{0,...,n_1-2\}$, and $C_{1,n_1-1}$ into $C_{1,0}$. 
\end{itemize} 
In summary, we removed the periodic orbit $O_1$ for $g_0$ and replaced it by the periodic circles $(C_{1,i})_{i=0}^{n_1-1}$ for $g_1$. Inductively, we obtain for each $k\geq 1$ a blowup $(S_k,\pi_k:S_k\to S_{k-1})$ and a homeomorphism $g_k\colon S_k\to S_{k}$ as follows. Once we have defined $S_{k-1}$, $\pi_{k-1}$, and  $g_{k-1}$, we obtain $(S_k,\pi_k)$ and $g_{k}$ by blowing-up (as above) the orbit $O_{k}$ into periodic circles $(C_{k,i})_{i=0}^{n_k-1}$ for $g_k$.
In this way, we have $\pi_k\circ g_k=g_{k-1}\circ \pi_{k}$ for every $k\geq 1$.
Each map $\pi_k$ is invertible everywhere except circles $(C_{k,i})_{i=0}^{n_k-1}$ which are collapsed to points forming periodic orbit $O_k$ for $g_{k-1}$.

Define the inverse limit space $$S_{\infty}=\{(y_k)\in\Pi_{k=0}^{\infty} S_k \; ;\; \pi_k(y_k)=y_{k-1}, \forall k\geq 1\}.$$
Note that we can view each $S_k\subset \mathbb{S}^2$ and $\pi_k\colon \mathbb{S}^2\to \mathbb{S}^2$, extending it in a natural way to ``holes'' complementing $S_k$ in $\mathbb{S}^2$. By results of Whyburn \cite{Wyburn}, the space $S_{\infty}$ is then a Sierpi\'nski carpet (see \cite{BO} for details). Define the map $G_A\colon S_{\infty}\to S_{\infty}$ by $$G_A((y_k)_{k\geq 0})=(g_k(y_k))_{k\geq 0}.$$
We will also need a standard result on measures induced on inverse limits (see \cite[Chapter V, Theorem~3.1]{Parth}) which we adapt to our context. Let $\eta_k\colon S_\infty \to S_k$ be the natural projection onto $k$-th coordinate.

\begin{theorem}\cite{Parth}\label{thm:induced}
If $\mu_k$ are Borel measures defined on $S_k$ such that for any Borel set $A\in \mathcal{B}(S_{k-1})$ we have
$\mu_{k-1}(A)=\mu_{k}(\pi_k^{-1}(A))$, then there exists a unique Borel measure $\nu$ on $S_\infty$ such that
$\nu(\eta_k^{-1}(A))=\mu_k(A)$ for every $k\in\N$ and every Borel set $A\in \mathcal{B}(S_{k})$.
\end{theorem}

Similarly to Theorem~\ref{mme} define
$$
\nu_n=\frac{1}{\Per_n(G_A)}\sum_{p\in P_n(G_A)}\delta_p.
$$
The following is the main result of this section.

\begin{theorem}
The homeomorphism $G_A$ of Sierpi\'nski carpet admits the unique measure of maximal entropy $\nu$.
Furthermore, the sequence $(\nu_n)_{n\in\N}$ converges to the measure $\nu$.

\end{theorem}

\begin{proof}
Let $\mu$ be the unique measure of maximal entropy of $g_A=g_0$ obtained in Theorem \ref{mme}. For each $k\in\N$, the map $\phi_k=\pi_k\circ \ldots \circ \pi_1\colon S_k\to \mathbb{S}^2$
satisfies $\phi_k\circ g_k=g_A\circ \phi_k$
and furthermore $\phi_k$ restricted to the set $S_k\setminus \phi_k ^{-1}(\Per(g_A))$ is one-to-one onto its image. But $\mu$ is ergodic, hence $\mu(\Per(g_A))=0$ and therefore there is a measure $\mu_k$ for $g_k$ which is isomorphic to $\mu$. Precisely, $\phi_k$ is a bi-measurable bijection from $S_k\setminus \phi_k^{-1}(Per(g_A))$ onto its image,  $(\phi_k)_*(\mu_k)=\mu$ and $\phi_k\circ g_k=g_A\circ \phi_k$.


Observe that $g_k^{n_k}|_{C_{k,i}}$ is a circle homeomorphism for every $k\in\N$ and $i\in\{0,\dots,n_k-1\}$. Since every periodic point of $f_A$ is a hyperbolic saddle and the projection $\pi\colon\mathbb{T}^2\to\mathbb{S}^2$ is a local isometry outside of the branching points, it follows that $g_A^{n_k}$ can fix at most $4$ radial lines in a neighborhood of any periodic point in $O_k$, precisely their stable and unstable spaces. Consequently, $g_k$ has exactly $4$ periodic points in $C_{k,i}$. In other words, the orbit $O_k$ of period $n_k$ for $g_A$ is replaced by $4n_k$ points in $\cup_i C_{k,i}$ of period $n_k$ for $g_k$. Furthermore, $g_k$ does not have any other periodic points in that set.
It is also clear that $\mu_k(C_{k,i})=0$ for every $i$. Therefore, Theorem~\ref{thm:induced} applies and we obtain a probability Borel measure $\nu$ on $S_\infty$.
Denote 
$$
C_\infty=\{x\in S_\infty ; \exists k\in \N\textrm{ such that }  x_k=\eta_k(x)\in \cup_{i=0}^{n_k} C_{k,i} \}
$$
and $M_\infty=S_\infty \setminus C_\infty$. Note that any point in $M_\infty$ is uniquely determined by any of its coordinates, that is, points in $M_\infty$ are of the form $(y,y,y,\dots)$. In particular, $\eta_k\colon M_\infty \to S_k$ is injective for every $k\geq 0$. This and Theorem \ref{thm:induced} ensure that
$$\nu(M_{\infty})=\mu_k(\eta_k(M_{\infty})) \,\,\,\,\,\, \text{for every} \,\,\,\,\,\, k\geq 0.$$  On the other hand,  $\phi_k(\eta_k(M_{\infty}))=S^2\setminus \cup_{k\in \N} O_k$ and therefore $\mu(\phi_k(\eta_k(M_{\infty})))=1$. Thus, $1=\mu_{k}(\eta_k(M_{\infty}))=\nu(M_{\infty})$.  But this implies $\nu(C_{\infty})=0$.

Since $\eta_0\colon M_\infty \to \mathbb{S}^2\setminus \bigcup_{k\in\N} O_k$ is injective, $M_\infty=S_\infty \setminus C_\infty$, and $\nu(C_{\infty})=0$, it follows that $\nu$ and $\mu$ are isomorphic, and in particular have the same entropy. 
Similarly, measures supported on periodic points outside of $C_\infty$ are isomorphically transported onto measures on corresponding periodic points of $g_A$.

Observe that if $y\in \eta_k^{-1}(C_{k,i})$
then $y=(y_0,\ldots, y_{k-1},y_k,y_k,y_k,\ldots)$
and
$$
G_A(y)=(g_0(y_0),\ldots, g_{k-1}(y_{k-1}),g_k(y_k),g_k(y_k),g_k(y_k),\ldots),$$ 
that is, the dynamics of $G_A$ on $\eta_k^{-1}(\cup_{i=1}^{n_k} C_{k,i})$ is conjugated to the dynamics of $g_k$ on $\cup_{i=1}^{n_k} C_{k,i}$, which as a homeomorphism of the circle has zero topological entropy. This shows that topological entropy satisfies $h(g_A)=h(G_A)$, in particular $\nu$ is a measure of maximal entropy. Repeating the previous argument, it is clear that any ergodic measure of positive entropy must assign zero measure to $C_\infty$, and hence any measure of maximal entropy is isomorphic to $\mu$ by projection onto the first coordinate and coincides with $\nu$.

Observe that $\Per_n(G_A|_{C_\infty})\leq \sum_{i=1}^n 4n\leq 4n^2$
while by Proposition~\ref{PergA} and the construction
we have 
$$
\Per_n(G_A)\geq \Per_n(g_A)\geq \lambda^n \,\,\,\,\,\, \text{and} \,\,\,\,\,\,
|\Per_n(G_A)-\Per_n(g_A)|\leq 4n^2.
$$
This immediately implies that if we consider 
$$
\tilde{\nu}_n=\frac{1}{\Per_n(G_A)\cap M_\infty}\sum_{p\in P_n(G_A)\cap M_\infty}\delta_p,
$$
then $\tilde{\nu}_n$ converges to $\nu$ in weak* topology, because each of these measures push-forward isomorphically to $(\eta_0)_*\tilde{\nu}_n$ converging to $\mu$. It is also clear that convergence of the sequence  $\tilde{\nu}_n$ implies convergence of the sequence ${\nu}_n$ and their weak* limits are equal. This completes the proof.


\end{proof}

\section{Further remarks and questions}

As one of the goals of our research, we tried to generalize the approach for the sphere to obtain a measure of maximal entropy in the general case of cw-hyperbolic homeomorphisms. Unfortunatelly, we were unable to adapt the proof of Bowen, even with additional hypotheses such as jointly continuous stable/unstable holonomies. The difference between expansiveness and cw-expansiveness is important to be observed. While in the expansive case the set of $n$-periodic points is finite and $(n,c)$-separated, where $c$ is an expansivity constant, in the cw-expansive case we could have even a Cantor set of periodic points that do not separate (as in the case of the identity in a Cantor set which is cw-expansive and has shadowing). This motivates the following question:

\begin{question}
Does joint continuity of stable/unstable holonomies imply  finiteness of periodic points with the same period?
\end{question}

Assuming cw$_N$-expansiveness, where there are at most $N$ distinct intersections among all pairs of local stable/unstable continua (see Definition 4 in \cite{ACCV3}), it is possible to prove finiteness of $k$-periodic points for every $k\in\N$ (this result was first obtained in \cite{ACCV3}*{Theorem 4.2}). The relation between joint continuity of holonomies and cw$_N$-expansiveness is unknown in full generality. It is proved in \cite{CR}*{Theorem 1.4} that cw$_F$-hyperbolic surface homeomorphisms with only regular byasymptotic sectors (or simply with a finite number of spines) have jointly continuous stable/unstable holonomies. The following is still an open question:

\begin{question}
Does joint continuity of stable/unstable holonomies imply cw$_N$- expansiveness for some $N\in\N$? And cw$_F$-expansiveness (see Definition 2.1 in \cite{ACS})?
\end{question}

Even the existence of a measure of maximal entropy in the general case of cw-hyperbolic homeomorphisms is still an open problem. When $f$ is cw$_N$-hyperbolic by results of \cite{ACCV3}, we can define for each $n\in\N$ the measure
$$
\mu_n=\frac{1}{\Per_n(f)}\sum_{p\in P_n(f)}\delta_p
$$
since $\Per_n(f)$ is finite for every $n\in\N$. Unfortunately, it is not easy to answer whether any measure that is a weak* limit of the sequence $(\mu_n)_{n\in\N}$ is a measure of maximal entropy, because the number of periodic points $\Per_n(f)$
can possibly grow much faster than $s_n(\delta)$, and not necessarily spread uniformly over the space. This motivates the following question.

\begin{question}
Does there exist a measure of maximal entropy for cw-hyperbolic homeomorphisms? If there is at least one, is it unique? If yes, is it obtained as the weak* limit of an average of Dirac measures on periodic orbits? Does Bowen's formula, as in Lemma \ref{formula} where topological entropy is obtained by the exponential growth of the number of $n$-periodic points, hold for cw-hyperbolic homeomorphisms? 
\end{question}

In Section 3, we answered this question affirmatively in the case of the pseudo-Anosov diffeomorphism $g_A$ of $\mathbb{S}^2$. The next question is regarding the measure of maximal entropy for $g_A$ obtained in Theorem \ref{mme}. Bowen's proof in the expansive case goes through proving that any measure accumulated by the measures $(\mu_n)_{n\in\N}$ is homogeneous (see Definition 19.6 in \cite{DGS} for a precise definition).

\begin{question}
Is the measure of maximal entropy $\mu$ obtained in Theorem \ref{mme} for $g_A$ homogeneous?
\end{question}

The last question is also related to continuity of stable/unstable holonomies and is motivated by Theorem~\ref{thm:tracing} and Remark \ref{expansiveandcw}. 

\begin{question}\label{5}
Is the constant $r$ necessary in Theorem~\ref{thm:tracing}, or in other words, does it hold with $r=1$?
\end{question}

\hspace{-0.45cm}\textbf{Acknowledgments.}
Research of Bernardo Carvalho was supported by Progetto di Eccellenza MatMod@TOV grant number CUP E83C23000330006, Prin 2022 (PRIN 2017S35EHN), and CNPq project number 446192/2024. Research of Piotr Oprocha was supported by National Science Centre, Poland (NCN), grant no. 2019/35/B/ST1/02239.

\begin{table}[h]
\begin{tabularx}{\linewidth}{p{1.5cm}  X}
\includegraphics [width=1.8cm]{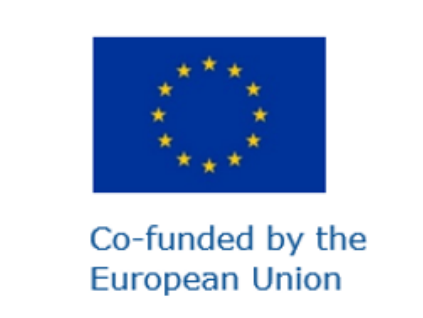} &
\vspace{-1.5cm}
This research is part of a project that has received funding from
the European Union's European Research Council Marie Sklodowska-Curie Project No. 101151716 -- TMSHADS -- HORIZON--MSCA--2023--PF--01.\\
\end{tabularx}
\end{table}

\vspace{0.4cm}
\noindent

{\em B. Carvalho}

\vspace{0.2cm}

\noindent

National Laboratory for Scientific Computing – LNCC/MCTI 

Av. Getúlio Vargas 333, CEP 25651-070, 

Petrópolis – RJ, Brazil

\vspace{0.2cm}

\email{bmcarvalho@lncc.br}

\vspace{0.2cm}

\noindent

Dipartimento di Matematica,

Università degli Studi di Roma Tor Vergata

Via Cracovia n.50 - 00133

Roma - RM, Italy
\vspace{0.2cm}

\email{mldbnr01@uniroma2.it}

\vspace{0.8cm}
\noindent

{\em P. Oprocha}
\vspace{0.2cm}

\noindent

AGH University of Krakow 

Faculty of Applied Mathematics

al. Mickiewicza 30 

30-059 Kraków, Poland

\vspace{+0.4cm}

Centre of Excellence IT4Innovations 

Institute for Research and Applications of Fuzzy Modeling 

University of Ostrava, 30. 

dubna 22, 701 03 Ostrava 1,

Czech Republic

\vspace{0.2cm}

\email{piotr.oprocha@osu.cz}

\vspace{0.8cm}
\noindent

{\em E. Rego}
\vspace{0.2cm}

AGH University of Krakow, 

Faculty of Applied Mathematics,

al. Mickiewicza 30,

30-059 Krak\'ow,

\vspace{0.2cm}

\email{rego@agh.edu.pl}

\end{document}